\documentclass[a4paper, 11pt, reqno]{amsart}
\usepackage{amsmath,amsfonts,amssymb,amsthm,enumerate}
\usepackage{graphicx}
\usepackage{xy} \xyoption{all}

\usepackage{xcolor}

\newtheorem{thm}{Theorem}[section]

\newtheorem{prop}[thm]{Proposition}
\newtheorem{lem}[thm]{Lemma}

\theoremstyle{definition}
\newtheorem{defn}[thm]{Definition}
\newtheorem{exas}[thm]{Example}
\newtheorem{rem}[thm]{Remark}

\let\phi\varphi

\title[Reduction theorem for Leavitt labelled path algebras]
{The Reduction Theorem for Leavitt Labelled Path Algebras and Its Applications}

\author[D. Gon\c{c}alves]{Daniel Gon\c{c}alves}
\address{Departamento de Matem\'atica, Universidade Federal de Santa Catarina,
88040-970 Florian\'opolis, SC, Brazil}
\email{daemig@gmail.com}

\author[N. D. Nam]{Nguyen Dinh Nam}
\address{Faculty of Pedagogy, Ha Tinh University, Ha Tinh, Vietnam}
\email{nam.nguyendinh@htu.edu.vn}

\author[T. G. Nam]{Tran Giang Nam}
\address{Institute of Mathematics, VAST, 18 Hoang Quoc Viet, Cau Giay,
Hanoi, Vietnam}
\email{tgnam@math.ac.vn}

\thanks{The first author was partially supported by Conselho Nacional de
Desenvolvimento Cient\'ifico e Tecnol\'ogico (CNPq), Brazil, and Funda\c{c}\~ao
de Amparo \`a Pesquisa e Inova\c{c}\~ao do Estado de Santa Catarina (FAPESC).
The second and third authors were supported by the Vietnam Institute for
Advanced Study in Mathematics (VIASM)}



\begin{document}

\begin{abstract}
We introduce a notion of labelled cycle for normal labelled spaces and prove a
reduction theorem for Leavitt labelled path algebras. We show that every
nonzero element can be reduced, by suitable left and right multiplication,
either to a nonzero scalar multiple of a projection or to a polynomial
supported on a labelled cycle without exits. This extends the classical
reduction theorem for Leavitt path algebras of directed graphs and its
analogues for ultragraph Leavitt path algebras and subshift algebras. As
applications, we prove the graded uniqueness theorem and the Cuntz--Krieger
uniqueness theorem for Leavitt labelled path algebras, and show that these
algebras are semiprime and semiprimitive over fields.
\end{abstract}

\maketitle

\section{Introduction}

Leavitt path algebras associated with directed graphs were independently
introduced by Abrams and Aranda Pino in \cite{ap:tlpaoag05} and by Ara,
Moreno, and Pardo in \cite{amp:nktfga}. Since then, they have developed into
a prominent area of research in algebra owing to their remarkable versatility.
Among their numerous applications, we highlight their deep connections with
symbolic dynamics and the theory of graph \(C^*\)-algebras (see, for example,
\cite{a:lpatfd,AAS,AbHaz25}), as well as with chip-firing games (see, e.g.,
\cite{AbHaz23,HazNam1,HazNam2}). One of the principal motivations for the
study of Leavitt path algebras is their ability to unify ideas from several areas
of mathematics. In particular, they provide a purely algebraic framework for
investigating phenomena that originally arose in functional analysis, especially
in the theory of graph \(C^*\)-algebras, while avoiding the use of analytic
techniques. Consequently, they offer an accessible setting for algebraists while
still preserving rich structural, dynamical, and categorical properties.

Labelled graphs arise naturally in several areas of mathematics and theoretical
computer science, including symbolic dynamics, automata theory, coding theory,
and operator algebras. In symbolic dynamics, labelled graphs provide concrete
presentations of shift spaces, particularly sofic shifts, where infinite sequences
are described through labelled paths in a graph. This connection has made
labelled graphs an important combinatorial tool for studying dynamical systems.
At the same time, labelled graphs are flexible enough to encode constructions
which go beyond ordinary directed graphs, and this makes them a natural
framework for studying classes of algebras that simultaneously generalize graph
algebras, ultragraph algebras, and algebras arising from symbolic dynamics.

Labelled graph \(C^*\)-algebras were introduced by Bates and Pask in
\cite{BatesPask:caolg}, and they exhibit several interesting connections with
symbolic dynamics (see, e.g., \cite{BatesPask:caolg}, \cite{BaCaPask}, and
\cite{CasWyk}). The \(C^*\)-algebraic theory of labelled spaces has also been
developed through inverse semigroups, tight spectra, partial crossed products,
and groupoid models; see, for instance,
\cite{BoavaCastroMortari:inverseSemigroups,BoavaCastroMortari:groupoidModels,CasWyk}.
These approaches show that labelled spaces carry not only combinatorial
information, but also a rich topological and dynamical structure. In
\cite{BCGW:lpaolg}, Boava, de Castro, Gon\c{c}alves, and van Wyk introduced
and studied the algebraic analogue of a labelled graph \(C^*\)-algebra, called a
\emph{Leavitt labelled path algebra}. These algebras properly generalize
Leavitt path algebras of directed graphs and ultragraphs, and they also
interact naturally with partial skew group rings and Steinberg algebras. Consequently,
Leavitt labelled path algebras provide a unified algebraic framework for the study of graph, ultragraph, and symbolic-dynamical phenomena.

A central tool in the structure theory of Leavitt path algebras is the
\emph{reduction theorem}. In its classical form, it says that every nonzero
element can be reduced, by multiplying on the left and on the right by suitable
algebra elements, to a much simpler element: either a nonzero scalar multiple
of a vertex projection or an element supported on a cycle without exits. This
principle was originally established for Leavitt path algebras of directed graphs
in \cite{AAS}, and it has proven to be extremely useful for characterizing
ring-theoretic properties such as uniqueness theorems, ideal-theoretic
properties, semiprimeness, and semiprimitivity. Analogous results have
subsequently been obtained in several related settings, including relative Cohn
path algebras \cite{CantoGon}, ultragraph Leavitt path algebras
\cite{gon:ratrt19}, and algebras associated with one-sided subshifts over
arbitrary alphabets \cite{BaCaGoRo}.

The importance of reduction theorems extends beyond their immediate ring-theoretic consequences. In the ultragraph setting, the reduction theorem
is a key ingredient in the study of representations, including the description of
faithful representations arising from branching systems \cite{gon:ratrt19}. In
the setting of subshift algebras, the reduction theorem yields uniqueness
results and structural consequences such as semiprimeness and
semiprimitivity \cite{BaCaGoRo}. More recently, reduction techniques have
also played an essential role in the study of the socle of subshift algebras and
in applications to subshift conjugacy: in \cite{GoncalvesRoyer:socleSubshift},
the socle and its grading are used to obtain invariants for conjugacy of
Ott--Tomforde--Willis subshifts and for isometric conjugacy of subshifts
constructed with the product topology. These applications illustrate that reduction theorems are not merely technical devices; rather, they provide a powerful framework for extracting dynamical and representation-theoretic information from
noncommutative algebras.

The main goal of this article is to establish a reduction theorem for Leavitt
labelled path algebras. This fills a natural gap in the existing theory, since a general
reduction theorem for this class has not yet been available. To obtain such a
result, we introduce a notion of labelled cycle adapted to labelled spaces. This
notion is weaker than the one used in \cite{BCGW:lpaolg}, and it is designed
to include both the usual notion of a cycle in ultragraphs and its analogue in the setting of labelled graphs. Using this definition, we prove that every nonzero element of a
Leavitt labelled path algebra can be reduced either to a nonzero scalar multiple of a projection or to a polynomial supported on a labelled cycle without exits.

As applications of our reduction theorem, we recover and prove uniqueness
theorems for Leavitt labelled path algebras, including a graded uniqueness
theorem and a Cuntz--Krieger uniqueness theorem. We also show that every
Leavitt labelled path algebra over a field is semiprime and semiprimitive.
These consequences demonstrate that the reduction theorem provides a robust
tool for the algebraic structure theory of labelled spaces and opens the way
for further applications to representation theory, ideal structure,
and dynamical invariants associated with labelled spaces.

The paper is organized as follows. In Section~2, we recall the necessary
preliminaries on labelled graphs, labelled spaces, and Leavitt labelled path
algebras, following \cite{BatesPask:caolg} and \cite{BCGW:lpaolg}. In
Section~3, we introduce the notion of labelled cycles
(Definition~\ref{def:cycle}), which is a weaker version of the concept
introduced in \cite{BCGW:lpaolg}, and establish a reduction theorem for
Leavitt labelled path algebras (Theorem~\ref{thm:reduction}).

In Section~4, we use Theorem~\ref{thm:reduction} to establish the graded
uniqueness theorem (Theorem~\ref{thm:graded-uniq}) and the Cuntz--Krieger
uniqueness theorem (Theorem~\ref{thm:Cuntz-Krieger}) for Leavitt labelled
path algebras. In particular, we show that every Leavitt labelled path algebra
is semiprime (Theorem~\ref{semiprime}) and semiprimitive
(Theorem~\ref{semiprimitive}). Our approach differs from the existing proofs
in the settings of directed graphs and ultragraphs.

\section{Preliminaries}

In this section, we recall the necessary concepts and notation concerning
labelled spaces and Leavitt labelled path algebras, following
\cite{BatesPask:caolg,BCGW:lpaolg}.

A directed graph \(E=(E^0,E^1,r,s)\) consists of a nonempty set \(E^0\) of
vertices, a set \(E^1\) of edges, and range and source maps $
r,s:E^1\longrightarrow E^0.$
A path of length \(n\geq1\) in \(E\) is a sequence
\(\lambda=\lambda_1\lambda_2\cdots\lambda_n\) of edges such that
\[
r(\lambda_i)=s(\lambda_{i+1})
\quad\text{for all } i=1,\ldots,n-1.
\]
We write \(|\lambda|=n\) for the length of \(\lambda\), and we regard vertices
as paths of length \(0\). The set of paths of length \(n\) is denoted by
\(E^n\), and
\[
E^*=\bigcup_{n\geq0}E^n.
\]
The set of infinite paths is denoted by \(E^\infty\).

A labelled graph consists of a graph \(E\) together with a surjective labelling
map
\[
\mathcal L:E^1\longrightarrow\mathcal A,
\]
where \(\mathcal A\) is a fixed nonempty set, called the alphabet, and whose
elements are called letters. We denote by \(\mathcal A^*\) the set of all finite
words over \(\mathcal A\), including the empty word \(\omega\), and by
\(\mathcal A^\infty\) the set of all infinite words over \(\mathcal A\). We
regard \(\mathcal A^*\) as a monoid under concatenation. Thus, if
\(\alpha\in\mathcal A^*\setminus\{\omega\}\) and \(n\in\mathbb N^*\), then
\(\alpha^n\) denotes the word obtained by concatenating \(\alpha\) with itself
\(n\) times, and \(\alpha^\infty\in\mathcal A^\infty\) denotes the infinite
concatenation of \(\alpha\).

The labelling map extends in the usual way to finite and infinite paths:
\[
\mathcal L:E^n\longrightarrow\mathcal A^*
\quad (n\geq1),
\qquad
\mathcal L:E^\infty\longrightarrow\mathcal A^\infty.
\]
For \(n\geq1\), let $\mathcal L^n=\mathcal L(E^n)$
be the set of labelled paths of length \(n\), and let
$
\mathcal L^\infty=\mathcal L(E^\infty)$
be the set of infinite labelled paths. We also regard \(\omega\) as a labelled
path of length \(0\), and set
\[
\mathcal L^{\geq1}=\bigcup_{n\geq1}\mathcal L^n,\qquad
\mathcal L^*=\mathcal L^{\geq1}\cup\{\omega\},\qquad
\mathcal L^{\leq\infty}=\mathcal L^*\cup\mathcal L^\infty.
\]

For \(\alpha\in\mathcal L^*\) and \(A\in\mathcal P(E^0)\), the relative range
of \(\alpha\) with respect to \(A\) is the set
\[
r(A,\alpha)=
\begin{cases}
\{r(\lambda)\mid \lambda\in E^*,\ \mathcal L(\lambda)=\alpha,\ s(\lambda)\in A\},
& \text{if } \alpha\in\mathcal L^{\geq1},\\[2mm]
A, & \text{if } \alpha=\omega.
\end{cases}
\]
The range of \(\alpha\), denoted by \(r(\alpha)\), is defined by $
r(\alpha)=r(E^0,\alpha).$
Thus \(r(\omega)=E^0\), and, if \(\alpha\in\mathcal L^{\geq1}\), then
\[
r(\alpha)=\{r(\lambda)\mid \lambda\in E^*,\ \mathcal L(\lambda)=\alpha\}.
\]
For \(A\subseteq E^0\), we also define
\[
\mathcal L(AE^1)
=
\{\mathcal L(e)\mid e\in E^1,\ s(e)\in A\}
=
\{a\in\mathcal A\mid r(A,a)\neq\emptyset\}.
\]

A labelled path \(\alpha\) is a beginning of a labelled path \(\beta\) if $\beta=\alpha\beta'$
for some labelled path \(\beta'\). Labelled paths \(\alpha\) and \(\beta\) are
called comparable if one is a beginning of the other. If
\(\alpha=\alpha_1\alpha_2\cdots\alpha_n\in\mathcal L^n\), then, for
\(1\leq i\leq j\leq n\), we define
$
\alpha_{i,j}=\alpha_i\alpha_{i+1}\cdots\alpha_j,
$
and we set \(\alpha_{i,j}=\omega\) whenever \(j<i\).

\begin{defn}[{\cite[page 4]{BCGW:lpaolg}}]\label{def:labelledspace}
A labelled space is a triple \((E,\mathcal L,\mathcal B)\), where
\((E,\mathcal L)\) is a labelled graph and \(\mathcal B\) is a family of subsets
of \(E^0\) which is closed under finite intersections and finite unions,
contains \(r(\alpha)\) for every \(\alpha\in\mathcal L^{\geq1}\), and is closed
under relative ranges, that is,
$
r(A,\alpha)\in\mathcal B
\quad
\text{for all } A\in\mathcal B \text{ and all } \alpha\in\mathcal L^*.
$
The family \(\mathcal B\) is called an accommodating family for
\((E,\mathcal L)\).

A labelled space \((E,\mathcal L,\mathcal B)\) is weakly left-resolving if, for
all \(A,B\in\mathcal B\) and all \(\alpha\in\mathcal L^{\geq1}\), we have
\[
r(A\cap B,\alpha)=r(A,\alpha)\cap r(B,\alpha).
\]
A weakly left-resolving labelled space such that \(\mathcal B\) is closed under
relative complements will be called normal.
\end{defn}

We next recall two standard examples of labelled spaces.

 \begin{exas}[{\cite[Example 3.3]{BatesPask:caolg}}]\label{exa:labelledspace}

\noindent \begin{enumerate}
\item Given a graph \(E\), let \(\mathcal A=E^1\), let
\(\mathcal L:E^1\to\mathcal A\) be the identity map, and let \(\mathcal B\) be
the set of all finite subsets of \(E^0\). Then \((E,\mathcal L,\mathcal B)\) is
a normal labelled space.

\item Recall that an ultragraph is a quadruple
$
\mathcal G=(G^0,\mathcal G^1,r,s)$
consisting of a set of vertices \(G^0\), a set of edges \(\mathcal G^1\), and
maps
$
s:\mathcal G^1\to G^0,$
$r:\mathcal G^1\to\mathcal P(G^0)\setminus\{\emptyset\};$
see, for example, \cite[Definition 2.1]{tomf:auatelaacaastg03}.

Given an ultragraph \(\mathcal G=(G^0,\mathcal G^1,r,s)\), we can build an
associated labelled space as follows. Let \(E=E_{\mathcal G}\) be the graph
with
\[
E^0=G^0,
\qquad
E^1=\{(e,w)\mid e\in\mathcal G^1,\ w\in r(e)\}.
\]
The source and range maps \(s',r':E^1\to E^0\) are defined by $
s'(e,w)=s(e),$ $r'(e,w)=w.$
Set \(\mathcal A=\mathcal G^1\) and define
$
\mathcal L:E^1\to\mathcal A,$ $
\mathcal L(e,w)=e.$
Let \(\mathcal G^0\) be the smallest subset of \(\mathcal P(G^0)\) containing
all singleton vertices and all ranges of edges, and closed under finite
intersections, finite unions, and relative complements. If we let
\(\mathcal B=\mathcal G^0\), then \((E,\mathcal L,\mathcal B)\) is a normal
labelled space.
\end{enumerate}
\end{exas}

For a labelled space \((E,\mathcal L,\mathcal B)\), a nonempty set
\(A\in\mathcal B\) is called regular if, for every nonempty
\(B\in\mathcal B\) with \(B\subseteq A\), we have
$0<|\mathcal L(BE^1)|<\infty.$
The set of all regular elements of \(\mathcal B\), together with the empty set,
is denoted by \(\mathcal B_{\mathrm{reg}}\). For \(\alpha\in\mathcal L^*\), define
\[
\mathcal B_\alpha
=
\mathcal B\cap\mathcal P(r(\alpha))
=
\{A\in\mathcal B\mid A\subseteq r(\alpha)\}.
\]
If the labelled space is normal, then \(\mathcal B_\alpha\) is a Boolean
algebra for each \(\alpha\in\mathcal L^*\).

\begin{defn}[{\cite[Definition 3.1]{BCGW:lpaolg}}]\label{Def}
Let \((E,\mathcal L,\mathcal B)\) be a normal labelled space and let \(K\) be a
field. The Leavitt labelled path algebra associated with
\((E,\mathcal L,\mathcal B)\) with coefficients in \(K\), denoted by
\(L_K(E,\mathcal L,\mathcal B)\), is the universal \(K\)-algebra with generators  $\{p_A\mid A\in\mathcal B\}$ and $\{s_a,s_a^*\mid a\in\mathcal A\},$
subject to the following relations:
\begin{enumerate}
\item
\(p_{A\cap B}=p_Ap_B\),
\(p_{A\cup B}=p_A+p_B-p_{A\cap B}\), and \(p_\emptyset=0\), for every
\(A,B\in\mathcal B\);

\item
\(p_As_a=s_ap_{r(A,a)}\) and
\(s_a^*p_A=p_{r(A,a)}s_a^*\), for every \(A\in\mathcal B\) and
\(a\in\mathcal A\);

\item
\(s_a^*s_a=p_{r(a)}\) and \(s_b^*s_a=0\) if \(b\neq a\), for every
\(a,b\in\mathcal A\);

\item
\(s_as_a^*s_a=s_a\) and \(s_a^*s_as_a^*=s_a^*\), for every
\(a\in\mathcal A\);

\item
$
p_A=\sum_{a\in\mathcal L(AE^1)}s_ap_{r(A,a)}s_a^*,
$
for every \(A\in\mathcal B_{\mathrm{reg}}\).
\end{enumerate}
\end{defn}

For \(\alpha=a_1a_2\cdots a_n\in\mathcal L^{\geq1}\), we write
\[
s_\alpha=s_{a_1}s_{a_2}\cdots s_{a_n}
\quad\text{and}\quad
s_\alpha^*=s_{a_n}^*\cdots s_{a_2}^*s_{a_1}^*.
\]
The assignments
\[
p_A\mapsto p_A,\qquad s_a\mapsto s_a^*,\qquad s_a^*\mapsto s_a
\]
define an involution on \(L_K(E,\mathcal L,\mathcal B)\). Although
\(L_K(E,\mathcal L,\mathcal B)\) is not necessarily unital, we also set $s_\omega=s_\omega^*=1,$
where \(\omega\) is the empty word. This convention is used only to simplify
notation. For instance, \(s_\omega p_A s_\omega^*\) means \(p_A\). We never use
\(s_\omega\) by itself as an element of the algebra.

We record the following examples.

\begin{exas}\label{exa:Labelledleavittpathalgebra}

\noindent \begin{enumerate}
\item
\cite[Example~7.1]{BCGW:lpaolg}
Let \(E\) be an arbitrary graph, let \((E,\mathcal L,\mathcal B)\) be the normal
labelled space described in Example~\ref{exa:labelledspace}(1), and let \(K\)
be a field. Let \(L_K(E)\) denote the Leavitt path algebra of \(E\); see, for
example, \cite[Definition~1.2.3]{AAS}. Then
\(L_K(E,\mathcal L,\mathcal B)\) is isomorphic to \(L_K(E)\) via the map
\[
p_A\longmapsto\sum_{v\in A}v,
\qquad
s_e\longmapsto e,
\qquad
s_e^*\longmapsto e^*.
\]

\item
\cite[Example~7.2]{BCGW:lpaolg}
Let \(\mathcal G\) be an ultragraph, let \((E,\mathcal L,\mathcal B)\) be the
normal labelled space described in Example~\ref{exa:labelledspace}(2), and let
\(K\) be a field. Let \(L_K(\mathcal G)\) be the Leavitt path algebra of
\(\mathcal G\); see \cite[Definition 2.3]{gr:saccfulpavpsgrt} and
\cite[Definition~2.1]{ima:tlpaou}. Then
\(L_K(E,\mathcal L,\mathcal B)\) is isomorphic to \(L_K(\mathcal G)\) via an
isomorphism sending generators to generators.
\end{enumerate}
\end{exas}

The following proposition collects basic properties of Leavitt labelled path
algebras that will be used throughout the paper.

\begin{prop}\label{properties}
Let \((E,\mathcal L,\mathcal B)\) be a normal labelled space and let \(K\) be a
field. Then \(L_K(E,\mathcal L,\mathcal B)\) has the following properties:
\begin{enumerate}
\item
{\rm\cite[Lemma 4.12]{BCGW:lpaolg}}
All elements of the set
\[
\{p_A,s_a,s_a^*\mid A\in\mathcal B\setminus\{\emptyset\},\ a\in\mathcal A\}
\]
are nonzero.

\item
{\rm\cite[Proposition 3.2]{BCGW:lpaolg}}
For any
\(\alpha,\beta,\gamma,\sigma\in\mathcal L^*\),
\(A\in\mathcal B_\alpha\cap\mathcal B_\beta\), and
\(B\in\mathcal B_\gamma\cap\mathcal B_\sigma\), we have
\[
(s_\alpha p_A s_\beta^*)(s_\gamma p_B s_\sigma^*)
=
\begin{cases}
s_{\alpha\gamma'}p_{r(A,\gamma')\cap B}s_\sigma^*,
& \text{if } \gamma=\beta\gamma',\\[1mm]
s_\alpha p_{A\cap r(B,\beta')}s_{\sigma\beta'}^*,
& \text{if } \beta=\gamma\beta',\\[1mm]
0,
& \text{otherwise.}
\end{cases}
\]

\item
{\rm\cite[Proposition 3.8]{BCGW:lpaolg}}
The algebra \(L_K(E,\mathcal L,\mathcal B)\) is linearly spanned by
\[
\{s_\alpha p_A s_\beta^*
\mid
\alpha,\beta\in\mathcal L^*,\ A\in\mathcal B_\alpha\cap\mathcal B_\beta\}.
\]
Furthermore, \(L_K(E,\mathcal L,\mathcal B)\) is a \(\mathbb Z\)-graded
\(K\)-algebra with grading
\[
L_K(E,\mathcal L,\mathcal B)_n
=
\operatorname{span}_K
\{s_\alpha p_A s_\beta^*
\mid
\alpha,\beta\in\mathcal L^*,\
A\in\mathcal B_\alpha\cap\mathcal B_\beta,\
|\alpha|-|\beta|=n\}
\]
for all \(n\in\mathbb Z\).
\end{enumerate}
\end{prop}

A monomial in \(L_K(E,\mathcal L,\mathcal B)\) is called a {\it real path} if it
contains no factor of the form \(s_a^*\), and it is called a {\it ghost path} if it
contains no factor of the form \(s_a\), for any \(a\in\mathcal A\). An element
\(x\in L_K(E,\mathcal L,\mathcal B)\) is said to be in {\it only real edges}
(respectively, in {\it only ghost edges}) if it is a \(K\)-linear combination of real
paths (respectively, ghost paths).

\section{The reduction theorem}

The main aim of this section is to establish the reduction theorem for
Leavitt labelled path algebras (Theorem~\ref{thm:reduction}). We begin with
a reduction for elements involving only real edges.

\begin{lem}\label{only real edges}
Let \((E,\mathcal L,\mathcal B)\) be a normal labelled space and \(K\) a field,
and let \(\alpha\in L_K(E,\mathcal L,\mathcal B)\) be a nonzero polynomial in
only real edges. Then there exist elements
\(a,b\in L_K(E,\mathcal L,\mathcal B)\), a nonempty set \(A\in\mathcal B\), and
a nonzero scalar \(k\in K\) such that
\[
a\alpha b
=
kp_A+\sum_{i=1}^n k_i s_{\alpha_i}p_{A_i},
\]
where \(k_i\in K\), \(\alpha_i\in\mathcal L^{\geq1}\), and
\(A_i\in\mathcal B\setminus\{\emptyset\}\), with
\(A_i\subseteq r(A,\alpha_i)\cap A\) for all \(1\leq i\leq n\).
\end{lem}

\begin{proof}
Since \(\alpha\) is a nonzero polynomial in only real edges, Proposition~\ref{properties}(3)
allows us to write
\[
\alpha
=
\sum_{i=1}^n h_i p_{B_i}
+
\sum_{j=1}^m k_j s_{\alpha_j}p_{A_j},
\]
where \(h_i,k_j\in K\), \(\alpha_j\in\mathcal L^{\geq1}\), and
\(B_i,A_j\in\mathcal B\setminus\{\emptyset\}\), with
\(A_j\subseteq r(\alpha_j)\) for all \(j\).

Suppose first that \(h_i\neq0\) for some \(i\). Let
\(J=\{i\mid h_i\neq0\}\). By \cite[Lemma 3.5]{BCGW:lpaolg}, there exist
nonzero scalars \(l_1,\ldots,l_s\in K\) and pairwise disjoint nonempty sets
\(C_1,\ldots,C_s\in\mathcal B\) such that
\[
\sum_{i\in J}h_i p_{B_i}=\sum_{i=1}^s l_i p_{C_i}.
\]
Thus
\[
p_{C_1}\alpha p_{C_1}
=
l_1p_{C_1}
+
\sum_{j=1}^m k_j s_{\alpha_j}
p_{A_j\cap C_1\cap r(C_1,\alpha_j)}.
\]
After discarding the zero terms, this has the desired form.

Now suppose that \(h_i=0\) for all \(i\). Then
\(\alpha=\sum_{j=1}^m k_j s_{\alpha_j}p_{A_j}\neq0\). After reordering and
removing zero coefficients, we may assume that
\begin{center}
\(0<|\alpha_1|\leq|\alpha_2|\leq\cdots\leq|\alpha_m|\) and \(k_j\neq0\) for
all \(j\). 
\end{center}
Multiplying on the left by \(s_{\alpha_1}^*\), Proposition~\ref{properties}(2)
gives
\[
s_{\alpha_1}^*s_{\alpha_j}
=
\begin{cases}
p_{r(\alpha_1)}s_{\alpha_j'}, & \text{if } \alpha_j=\alpha_1\alpha_j'
\text{ and } |\alpha_j'|\geq1,\\
p_{r(\alpha_1)}, & \text{if } \alpha_j=\alpha_1,\\
0, & \text{otherwise.}
\end{cases}
\]
Hence, after reindexing,
\[
s_{\alpha_1}^*\alpha
=
\sum_{j=1}^{m_1} k_j p_{A_j}
+
\sum_{j=1}^{m_2} k_j s_{\alpha_j'}p_{A_j},
\]
where the first sum corresponds to the indices with \(\alpha_j=\alpha_1\), and
the second to those with \(\alpha_j=\alpha_1\alpha_j'\), \(|\alpha_j'|\geq1\).
Applying \cite[Lemma 3.5]{BCGW:lpaolg} to the projection part, there exist
nonzero scalars \(g_1,\ldots,g_t\in K\) and pairwise disjoint nonempty sets
\(D_1,\ldots,D_t\in\mathcal B\) such that
\[
\sum_{j=1}^{m_1}k_jp_{A_j}=\sum_{i=1}^t g_i p_{D_i}.
\]
Therefore
\[
p_{D_1}s_{\alpha_1}^*\alpha p_{D_1}
=
g_1p_{D_1}
+
\sum_{j=1}^{m_2} k_j s_{\alpha_j'}
p_{A_j\cap D_1\cap r(D_1,\alpha_j')}.
\]
After discarding the zero terms, this also has the desired form.
\end{proof}

We next introduce the notion of labelled cycles that will be used in the
reduction argument.

\begin{defn}\label{def:cycle}
Let \((E,\mathcal L,\mathcal B)\) be a normal labelled space.
\begin{enumerate}
\item
A pair \((\alpha,A)\), with \(\alpha\in\mathcal L^{\geq1}\) and
\(A\in\mathcal B_\alpha\), is called a {\it cycle} if, for every
\(B\in\mathcal B_\alpha\) with \(B\subseteq A\), one has
\(B\subseteq r(B,\alpha)\).

\item
A cycle \((\alpha,A)\) has an {\it exit} if there exist \(0\leq k\leq|\alpha|\) and
a nonempty set \(B\in\mathcal B\) such that
\(B\subseteq r(A,\alpha_{1,k})\) and
\(\mathcal L(BE^1)\neq\{\alpha_{k+1}\}\), where
\(\alpha_{|\alpha|+1}=\alpha_1\).
\end{enumerate}
\end{defn}

\begin{rem}\label{rem:cycle}
The notion of labelled cycles in labelled spaces was introduced in
\cite[Definition 8.1]{BCGW:lpaolg}. There, a pair \((\alpha,A)\), with
\(\alpha\in\mathcal L^{\geq1}\) and \(A\in\mathcal B_\alpha\), is a cycle if
\(B=r(B,\alpha)\) for every \(B\in\mathcal B_\alpha\) with \(B\subseteq A\).
This notion encompasses the usual notion of cycles in directed graphs; see,
for example, \cite[Definitions 2.0.2]{AAS}. However, it does not cover the
notion of cycles in ultragraphs; see, for example, \cite[page 6]{NN:pisulpa}.
The definition above is weaker and covers both the graph and ultragraph
notions of cycle.
\end{rem}

To clarify Definition \ref{def:cycle} and Remark \ref{rem:cycle}, we present the following example.

\begin{exas}\label{exa:weak-cycle}
Let \(E\) be the labelled graph with vertices $
E^0=\{v,w\},$
two edges \(e_1,e_2\), both labelled by \(a\), with
\[
s(e_1)=s(e_2)=v,\qquad r(e_1)=v,\qquad r(e_2)=w.
\]
Let \(\mathcal A=\{a\}\), let \(\mathcal L(e_1)=\mathcal L(e_2)=a\), and take
\(\mathcal B=\mathcal P(E^0)\). Then \((E,\mathcal L,\mathcal B)\) is a normal
labelled space.

Consider \(A=\{v\}\). We have
 $r(A,a)=\{v,w\}.$
Thus \(A\subseteq r(A,a)\), and since the only nonempty subset of \(A\) is
\(A\) itself, the pair \((a,A)\) is a cycle in the sense of
Definition~\ref{def:cycle}. However,
$
r(A,a)=\{v,w\}\neq A.$
Therefore \((a,A)\) is not a cycle in the sense of
\cite[Definition 8.1]{BCGW:lpaolg}. This illustrates that
Definition~\ref{def:cycle} is strictly weaker than the cycle definition used in
\cite{BCGW:lpaolg}.
\end{exas}

The following elementary observation will be used repeatedly.

\begin{lem}\label{cycle}
Let \((E,\mathcal L,\mathcal B)\) be a normal labelled space and let
\((\alpha,A)\) be a cycle. Then:
\begin{enumerate}
\item \((\alpha^n,A)\) is a cycle for every \(n\geq1\);
\item \(p_As_\alpha^n p_A=s_\alpha^n p_A\) for every \(n\geq1\).
\end{enumerate}
\end{lem}

\begin{proof}
We prove (1) by induction on \(n\). The case \(n=1\) is immediate. Suppose
that \((\alpha^{n-1},A)\) is a cycle for some \(n\geq2\). Since
\(A\subseteq r(A,\alpha^{n-1})\) and \(A\subseteq r(A,\alpha)\), we have
\[
A\subseteq r(A,\alpha)\subseteq r(r(A,\alpha^{n-1}),\alpha)=r(A,\alpha^n).
\]
Let \(B\in\mathcal B_{\alpha^n}\) be nonempty with \(B\subseteq A\). Since
\(A\in\mathcal B_{\alpha^{n-1}}\cap\mathcal B_\alpha\), we have
\(B\in\mathcal B_{\alpha^{n-1}}\cap\mathcal B_\alpha\). Hence the cycle
properties of \((\alpha^{n-1},A)\) and \((\alpha,A)\) give
\(B\subseteq r(B,\alpha^{n-1})\) and \(B\subseteq r(B,\alpha)\). Therefore
\[
B\subseteq r(B,\alpha)
\subseteq r(r(B,\alpha^{n-1}),\alpha)
=
r(B,\alpha^n).
\]
Thus \((\alpha^n,A)\) is a cycle.

For (2), by (1), \(A\subseteq r(A,\alpha^n)\) for every \(n\geq1\). Hence
\[
p_As_\alpha^n p_A
=
s_\alpha^n p_{r(A,\alpha^n)\cap A}
=
s_\alpha^n p_A.
\]
\end{proof}

The next lemma is the key periodicity fact needed to handle powers of a cycle.

\begin{lem}\label{gcd cycle}
Let \((E,\mathcal L,\mathcal B)\) be a normal labelled space, let
\(A\in\mathcal B\) be nonempty, and let \(\alpha\in\mathcal L^{\geq1}\).
Suppose that \(n\geq2\) and \(l_1,\ldots,l_n\) are positive integers such that
\((\alpha^{l_i},A)\) is a cycle for every \(i=1,\ldots,n\). Then
\((\alpha^l,A)\) is a cycle, where \(l=\gcd(l_1,\ldots,l_n)\).
\end{lem}

\begin{proof}
We prove the result by induction on \(n\). First consider the case \(n=2\).
Let \(l=\gcd(l_1,l_2)\). Without loss of generality, there exist positive
integers \(s,t\) such that \(sl_1+l=tl_2\).

Suppose that there exists a nonempty \(B\in\mathcal B\), with \(B\subseteq A\),
such that \(B\nsubseteq r(B,\alpha^l)\). Put \(D=B\setminus r(B,\alpha^l)\).
Then \(D\in\mathcal B\), \(D\neq\emptyset\), \(D\subseteq B\subseteq A\), and
\(D\cap r(B,\alpha^l)=\emptyset\). By Lemma~\ref{cycle},
\((\alpha^{sl_1},A)\) is a cycle, so \(D\subseteq r(D,\alpha^{sl_1})\).
Since the labelled space is weakly left-resolving,
\[
r(D,\alpha^{sl_1})\cap r(r(B,\alpha^l),\alpha^{sl_1})=\emptyset.
\]
But \(r(r(B,\alpha^l),\alpha^{sl_1})=r(B,\alpha^{l+sl_1})
=r(B,\alpha^{tl_2})\). On the other hand, \((\alpha^{tl_2},A)\) is a cycle,
again by Lemma~\ref{cycle}, so \(B\subseteq r(B,\alpha^{tl_2})\). Thus
\(D\subseteq r(D,\alpha^{sl_1})\cap r(B,\alpha^{tl_2})\), a contradiction.
Therefore \(B\subseteq r(B,\alpha^l)\) for every nonempty
\(B\in\mathcal B\) with \(B\subseteq A\). Taking \(B=A\), we obtain
\(A\subseteq r(A,\alpha^l)\subseteq r(\alpha^l)\), and hence
\(A\in\mathcal B_{\alpha^l}\). Thus \((\alpha^l,A)\) is a cycle.

For the induction step, assume the result holds for \(n-1\) integers and
suppose that \((\alpha^{l_i},A)\) is a cycle for \(i=1,\ldots,n\). Let
\(d=\gcd(l_1,\ldots,l_{n-1})\). By the induction hypothesis,
\((\alpha^d,A)\) is a cycle. Applying the two-variable case to \(d\) and
\(l_n\), we obtain that \((\alpha^{\gcd(d,l_n)},A)\) is a cycle. Since
\(\gcd(d,l_n)=\gcd(l_1,\ldots,l_n)\), the result follows, thus finishing the proof.
\end{proof}

Using Lemmas~\ref{only real edges}, \ref{cycle}, and \ref{gcd cycle}, we now
prove the real-edge case of the reduction theorem.

\begin{lem}\label{d=0}
Let \((E,\mathcal L,\mathcal B)\) be a normal labelled space and \(K\) a field,
and let \(\alpha\in L_K(E,\mathcal L,\mathcal B)\) be a nonzero polynomial in
only real edges. Then there exist elements \(a,b\in L_K(E,\mathcal L,\mathcal B)\)
and a nonempty set \(A\in\mathcal B\) such that either:
\begin{enumerate}
\item \(0\neq a\alpha b=kp_A\) for some \(k\in K\setminus\{0\}\); or
\item \(0\neq a\alpha b=\sum_{i=0}^n k_i s_c^i p_A\), where
\(n\in\mathbb N\), \(k_i\in K\), \((c,A)\) is a cycle without exit, and
\(s_c^0:=p_A\).
\end{enumerate}
\end{lem}

\begin{proof}
By Lemma~\ref{only real edges}, there exist
\(a',b'\in L_K(E,\mathcal L,\mathcal B)\), a nonempty \(A\in\mathcal B\), and
\(0\neq k\in K\) such that
\[
a'\alpha b'=kp_A+\sum_{i=1}^n k_i s_{\alpha_i}p_{A_i},
\]
where \(A_i\subseteq r(A,\alpha_i)\cap A\). If the second summand is zero,
there is nothing to prove. Otherwise, remove the zero terms.

Choose a maximal nonempty subset \(J\subseteq\{1,\ldots,n\}\) such that
\(M=\bigcap_{i\in J}A_i\neq\emptyset\) and
\(M\cap A_j=\emptyset\) for every \(j\notin J\). Then \(M\subseteq A\). If
\(M\cap r(M,\alpha_i)=\emptyset\) for all \(i\), then
\(p_M(a'\alpha b')p_M=kp_M\), and we are done. Otherwise, after removing zero
terms and reindexing, we may write
\[
p_M(a'\alpha b')p_M
=
kp_M+\sum_{i=1}^m k_i s_{\alpha_i}p_{M\cap r(M,\alpha_i)},
\]
where \(k_i\neq0\) and \(M\cap r(M,\alpha_i)\neq\emptyset\) for all \(i\).

Suppose that, for some \(j\), there exists a nonempty \(B\in\mathcal B\) with
\(B\subseteq M\) and \(B\nsubseteq r(B,\alpha_j)\). Put
\(C=B\setminus r(B,\alpha_j)\). Then \(C\neq\emptyset\), \(C\subseteq M\), and
\(r(C,\alpha_j)\cap C=\emptyset\). Hence
\[
p_Cs_{\alpha_j}p_{M\cap r(M,\alpha_j)}p_C
=
s_{\alpha_j}p_{r(C,\alpha_j)\cap C}
=
0.
\]
Thus the term corresponding to \(\alpha_j\) is eliminated. Repeating this
procedure finitely many times, we obtain a nonempty set \(H\in\mathcal B\),
\(H\subseteq M\), such that either \(p_H(a'\alpha b')p_H=kp_H\), or
\[
p_H(a'\alpha b')p_H
=
kp_H+\sum_{i=1}^{m_1}k_i s_{\alpha_i}p_{H\cap r(H,\alpha_i)},
\]
where all displayed terms are nonzero and, for every \(i\) and every nonempty
\(B\in\mathcal B\) with \(B\subseteq H\), one has \(B\subseteq r(B,\alpha_i)\).
In particular, each \((\alpha_i,H)\) is a cycle. If
\(p_H(a'\alpha b')p_H=kp_H\), we are done. Otherwise, since
\(H\subseteq r(H,\alpha_i)\) for every \(i\), we may write
\[
p_H(a'\alpha b')p_H
=
kp_H+\sum_{i=1}^{m_1}k_i s_{\alpha_i}p_H,
\]
where \(0<|\alpha_1|\leq|\alpha_2|\leq\cdots\leq|\alpha_{m_1}|\),
\(k_i\neq0\), and each \((\alpha_i,H)\) is a cycle.

Let \(r:=|\alpha_{m_1}|\,|\alpha_1|\). Then
\[
p_H(s_{\alpha_1}^*)^r p_H(a'\alpha b')p_Hs_{\alpha_1}^rp_H
=
kp_H+
\sum_{i=1}^{m_1}
k_i p_H(s_{\alpha_1}^*)^rs_{\alpha_i}p_Hs_{\alpha_1}^rp_H.
\]
Write \(\alpha_1=\beta^{l_1}\), where \(|\beta|\) is minimal and \(l_1\geq1\).
We claim that, for every \(i\), either \(\alpha_i=\beta^{l_i}\) for some
\(l_i\geq1\), or
\[
k_i p_H(s_{\alpha_1}^*)^rs_{\alpha_i}p_Hs_{\alpha_1}^rp_H=0.
\]

Indeed, suppose that
\(k_i p_H(s_{\alpha_1}^*)^rs_{\alpha_i}p_Hs_{\alpha_1}^rp_H\neq0\). Since
\(\alpha_1=\beta^{l_1}\), this implies
\[
p_H(s_\beta^*)^{rl_1}s_{\alpha_i}s_\beta^{rl_1}
p_{H\cap r(H,\alpha_i\beta^{rl_1})}
\neq0.
\]
Since \(|\alpha_i|\leq|\alpha_{m_1}|<r|\alpha_1|=rl_1|\beta|\), Proposition~\ref{properties}(2)
implies that \(\alpha_i\) is an initial segment of \(\beta^{rl_1}\). In
particular, the first \(|\beta|\) letters of \(\alpha_i\) are precisely
\(\beta\), and hence \(s_\beta^*s_{\alpha_i}\neq0\). Thus
\(\alpha_i=\beta^t\alpha_i^{(1)}\), for some \(t\geq1\), where
\(0\leq|\alpha_i^{(1)}|<|\beta|\).

If \(|\alpha_i^{(1)}|=0\), then \(\alpha_i=\beta^t\), and we are done. Assume
therefore that \(|\alpha_i^{(1)}|>0\). Then
\[
p_H(s_\beta^*)^{rl_1}s_{\alpha_i}s_\beta^{rl_1}
p_{H\cap r(H,\alpha_i\beta^{rl_1})}
=
p_H(s_\beta^*)^{rl_1-t}s_{\alpha_i^{(1)}}s_\beta^{rl_1}
p_{H\cap r(H,\alpha_i\beta^{rl_1})}
\neq0,
\]
and consequently \(s_\beta^*s_{\alpha_i^{(1)}}\neq0\). Hence
\(\beta=(\alpha_i^{(1)})^{t_1}\alpha_i^{(2)}\), for some \(t_1\geq1\), where
\(0\leq|\alpha_i^{(2)}|<|\alpha_i^{(1)}|\). If
\(|\alpha_i^{(2)}|=0\), then \(\beta=(\alpha_i^{(1)})^{t_1}\), and therefore
\[
\alpha_i=\beta^t\alpha_i^{(1)}
=
((\alpha_i^{(1)})^{t_1})^t\alpha_i^{(1)}
=
(\alpha_i^{(1)})^{t_1t+1},
\]
contradicting the minimality of \(|\beta|\). Thus \(|\alpha_i^{(2)}|>0\).
Moreover, the preceding nonzero product can be rewritten as
\[
p_H(s_\beta^*)^{rl_1-t-1}
s_{\alpha_i^{(2)}}^*
(s_{\alpha_i^{(1)}}^*)^{t_1}
s_{\alpha_i^{(1)}}s_{\alpha_i^{(1)}}^{t_1}
s_{\alpha_i^{(2)}}s_\beta^{rl_1-1}
p_{H\cap r(H,\alpha_i\beta^{rl_1})}
\neq0,
\]
and hence \(s_{\alpha_i^{(2)}}^*s_{\alpha_i^{(1)}}\neq0\). Thus
\(\alpha_i^{(1)}=(\alpha_i^{(2)})^{t_2}\alpha_i^{(3)}\), for some
\(t_2\geq1\), where \(0\leq|\alpha_i^{(3)}|<|\alpha_i^{(2)}|\).

Continuing inductively, after finitely many steps we obtain
\(\alpha_i^{(j-1)}=(\alpha_i^{(j)})^{t_j}\) for some \(j\). Tracing back
through the previous equalities yields \(\alpha_i=(\alpha_i^{(j)})^s\) for
some positive integer \(s\), again contradicting the minimality of \(|\beta|\).
Therefore, either \(\alpha_i=\beta^{l_i}\) for some \(l_i\geq1\), or
\(k_i p_H(s_{\alpha_1}^*)^rs_{\alpha_i}p_Hs_{\alpha_1}^rp_H=0\), as claimed.

It follows that, after discarding the zero terms,
\[
p_H(s_{\alpha_1}^*)^r p_H(a'\alpha b')p_Hs_{\alpha_1}^rp_H
=
kp_H+\sum_{i=1}^{m_2}k_i s_\beta^{l_i}p_H,
\]
where \(m_2\leq m_1\), \(l_i\in\mathbb N^*\), \(k_i\neq0\), and
\((\beta^{l_i},H)\) is a cycle for every \(i\). Let
\(l=\gcd(l_1,\ldots,l_{m_2})\). By Lemma~\ref{gcd cycle},
\((\beta^l,H)\) is a cycle. Since \(l\mid l_i\), say \(l_i=ll_i'\), the last
expression becomes
\[
kp_H+\sum_{i=1}^{m_2}k_i(s_\beta^l)^{l_i'}p_H.
\]
If \((\beta^l,H)\) has no exit, then the conclusion follows by taking
\(c=\beta^l\).

It remains to handle the case where \((\beta^l,H)\) has an exit. Write
\(\beta^l=a_1a_2\cdots a_{|\beta^l|}\), with
\(a_{|\beta^l|+1}=a_1\). By the definition of exit, there exist
\(0\leq q\leq|\beta^l|\) and a nonempty \(B\in\mathcal B\) such that
\(B\subseteq r(H,z)\), where \(z=a_1\cdots a_q\) if \(q\geq1\) and
\(z=\omega\) if \(q=0\), and \(\mathcal L(BE^1)\neq\{a_{q+1}\}\). If
\(\mathcal L(BE^1)=\emptyset\), then multiplying by \(p_Bs_z^*\) on the left
and by \(s_zp_B\) on the right gives \(kp_B\). Otherwise, choose
\(b\in\mathcal L(BE^1)\) with \(b\neq a_{q+1}\); multiplying by
\(s_b^*p_Bs_z^*\) on the left and by \(s_zp_Bs_b\) on the right gives
\(kp_{r(B,b)}\). In either case, we obtain a nonzero scalar multiple of a
projection, and the proof is complete.
\end{proof}

We now pass from real-edge polynomials to arbitrary elements. For this, we use
the following ghost-degree filtration.

A monomial
\[
s_{a_1}\cdots s_{a_m}p_As_{b_1}^*\cdots s_{b_n}^*
\]
in \(L_K(E,\mathcal L,\mathcal B)\), where \(a_i,b_j\in\mathcal A\), is said
to have ghost degree \(n\). If an element is written in the form
\(\sum_i k_i s_{\alpha_i}p_{A_i}s_{\beta_i}^*\), with \(k_i\neq0\), the ghost
degree of this expression is the maximum of the lengths \(|\beta_i|\). Since
an element of \(L_K(E,\mathcal L,\mathcal B)\) may have many such expressions,
we define the ghost degree of an element \(x\), denoted \(\operatorname{gdeg}(x)\),
to be the minimum ghost degree among all expressions of \(x\) in the spanning
form of Proposition~\ref{properties}(3).

We are now in a position to state the main theorem of this section, extending
\cite[Theorem 2.2.11]{AAS} and \cite[Theorem 3.2]{gon:ratrt19} to the
labelled graph setting.

\begin{thm}[The Reduction Theorem]\label{thm:reduction}
Let \((E,\mathcal L,\mathcal B)\) be a normal labelled space and \(K\) a field.
For every nonzero element \(\alpha\in L_K(E,\mathcal L,\mathcal B)\), there
exist \(a,b\in L_K(E,\mathcal L,\mathcal B)\) such that either:
\begin{enumerate}
\item \(0\neq a\alpha b=kp_A\) for some \(k\in K\setminus\{0\}\) and
\(A\in\mathcal B\setminus\{\emptyset\}\); or
\item \(0\neq a\alpha b=\sum_{i=0}^n k_i s_c^i p_A\), where
\(n\in\mathbb N\), \(k_i\in K\), \((c,A)\) is a cycle without exit, and
\(s_c^0:=p_A\).
\end{enumerate}
\end{thm}

\begin{proof}
We argue by induction on \(d=\operatorname{gdeg}(\alpha)\). If \(d=0\), the
result follows from Lemma~\ref{d=0}. Suppose \(d\geq1\), and assume the result
holds for all nonzero elements of ghost degree strictly smaller than \(d\).
Using Proposition~\ref{properties}(3), write
\[
\alpha
=
\sum_{t=1}^q g_t p_{C_t}
+
\sum_{j=1}^m h_j s_{\gamma_j}p_{B_j}
+
\sum_{i=1}^n x_i s_{a_i}^*,
\]
where \(n\geq1\), \(x_i\neq0\), \(a_i\in\mathcal A\), the letters \(a_i\) are
pairwise distinct, and \(\operatorname{gdeg}(x_i)<d\) for every \(i\).

Suppose first that \(\alpha p_A=0\) for every \(A\in\mathcal B\). Then
\(\alpha s_a^*=\alpha p_{r(a)}s_a^*=0\) for every \(a\in\mathcal A\). If also
\(\alpha s_a=0\) for every \(a\in\mathcal A\), then, since
\(L_K(E,\mathcal L,\mathcal B)\) has local units by
\cite[Remark 6.8]{BCGW:lpaolg}, we would have \(\alpha=0\), a contradiction.
Therefore \(\alpha s_a\neq0\) for some \(a\in\mathcal A\). Multiplication by
\(s_a\) on the right lowers the ghost degree, and hence
\(\operatorname{gdeg}(\alpha s_a)<d\). The induction hypothesis applies to
\(\alpha s_a\).

Now assume that \(\alpha p_A\neq0\) for some \(A\in\mathcal B\). After removing
zero terms, write
\[
\alpha p_A
=
\beta p_A+\sum_{i=1}^r x_i s_{a_i}^*p_A\neq0,
\]
where \(\operatorname{gdeg}(\beta)=0\), \(1\leq r\leq n\), the letters \(a_i\)
are pairwise distinct, and \(\operatorname{gdeg}(x_i)<d\). If
\(\operatorname{gdeg}(\alpha p_A)<d\), the induction hypothesis applies. Thus
we may assume that \(\operatorname{gdeg}(\alpha p_A)=d\). In particular,
\(\mathcal L(AE^1)\neq\emptyset\).

If \(\alpha p_A s_{a_j}\neq0\) for some \(j\), then
\[
\alpha p_A s_{a_j}
=
\beta p_A s_{a_j}+x_jp_{r(A,a_j)}
\]
has ghost degree strictly smaller than \(d\), and induction applies. We may
therefore assume that \(\alpha p_A s_{a_j}=0\) for all \(j=1,\ldots,r\). Then
\(x_jp_{r(A,a_j)}=-\beta p_As_{a_j}\), and substitution gives
\[
\alpha p_A
=
\beta p_A
\left(
p_A-\sum_{i=1}^r s_{a_i}p_{r(A,a_i)}s_{a_i}^*
\right)
\neq0.
\]
In particular, \(\beta p_A\neq0\) and the second factor is nonzero.

Write
\[
\beta p_A=\sum_{t=1}^q g_t p_{C_t}+\sum_{j=1}^m h_j s_{\gamma_j}p_{B_j},
\]
where \(C_t\subseteq A\) and \(B_j\subseteq A\cap r(\gamma_j)\). Let
\(B=(\bigcup_t C_t)\cup(\bigcup_j B_j)\). Then
\(B\in\mathcal B\setminus\{\emptyset\}\), \(B\subseteq A\), and
\(\beta p_B=\beta p_A\). Hence \(\alpha p_B=\alpha p_A\neq0\).

Partition \(B\) into finitely many nonempty Boolean pieces
\(H_i\in\mathcal B\) such that each \(C_t\) and each \(B_j\) either contains
\(H_i\) or is disjoint from \(H_i\). Since \(\alpha p_B\neq0\), there is a
piece \(H_k\) such that \(\alpha p_{H_k}\neq0\). For this piece,
\[
\alpha p_{H_k}
=
\beta p_{H_k}
\left(
p_{H_k}
-
\sum_{i=1}^r s_{a_i}p_{r(H_k,a_i)}s_{a_i}^*
\right)
\neq0.
\]
Thus \(\beta p_{H_k}\neq0\) and the second factor is nonzero.

If \(\mathcal L(H_kE^1)=\emptyset\), then
\(\alpha p_{H_k}=\beta p_{H_k}\), which has ghost degree \(0\), and induction
applies. Assume now that \(\mathcal L(H_kE^1)\neq\emptyset\). If there exists
a nonempty \(X\in\mathcal B\) with \(X\subseteq H_k\) and
\(\mathcal L(XE^1)=\emptyset\), then \(s_a^*p_X=0\) for all
\(a\in\mathcal A\), and hence \(\alpha p_X=\beta p_X\). We claim that
\(\beta p_X\neq0\). Since \(H_k\) is a Boolean piece, every coefficient set
\(C_t\) or \(B_j\) either contains \(H_k\) or is disjoint from it. Therefore
every term that survives after multiplication by \(p_{H_k}\) also survives
after multiplication by \(p_X\). Writing
\(\beta p_{H_k}=gp_{H_k}+\sum_{j=1}^l h_js_{\gamma_j}p_{H_k}\), if
\(g\neq0\), then \(gp_X\) is a nonzero degree-zero component of \(\beta p_X\).
If the real-path part is nonzero, then by the \(\mathbb Z\)-grading, after
combining equal labelled words if necessary, there exists a nonempty subset
\(J\subseteq\{1,\ldots,l\}\) such that the words \(\gamma_j\), \(j\in J\), are
pairwise distinct, have the same length, and
\(\sum_{j\in J}h_js_{\gamma_j}p_{H_k}\neq0\). If
\(\sum_{j\in J}h_js_{\gamma_j}p_X=0\), then multiplying on the left by
\(s_{\gamma_{i_0}}^*\), for some \(i_0\in J\) with \(h_{i_0}\neq0\), gives
\(h_{i_0}p_X=0\), a contradiction. Hence \(\beta p_X\neq0\). Therefore
\(\alpha p_X=\beta p_X\neq0\), and induction applies.

Accordingly, we may assume that \(\mathcal L(XE^1)\neq\emptyset\) for every
nonempty \(X\in\mathcal B\) with \(X\subseteq H_k\). If \(H_k\) is not regular,
then some nonempty \(X\in\mathcal B\), \(X\subseteq H_k\), emits infinitely
many labels. Hence \(\mathcal L(H_kE^1)\) is infinite, and we may choose
\(a\in\mathcal L(H_kE^1)\setminus\{a_1,\ldots,a_r\}\). If \(H_k\) is regular,
then the nonzero element
\(p_{H_k}-\sum_i s_{a_i}p_{r(H_k,a_i)}s_{a_i}^*\), together with the
Cuntz--Krieger relation for \(H_k\), also gives a label
\(a\in\mathcal L(H_kE^1)\setminus\{a_1,\ldots,a_r\}\). In either case,
\[
\alpha p_{H_k}s_a=\beta p_{H_k}s_a.
\]

It remains to prove that \(\beta p_{H_k}s_a\neq0\). Write \(\beta p_{H_k}\) as
the sum of its homogeneous real-degree components:
\[
\beta p_{H_k}
=
g'_kp_{H_k}
+
\sum_{j=1}^s\sum_{i=1}^{n_j}
k_i^{(j)}s_{\gamma_i^{(j)}}p_{H_k},
\]
where, for each fixed \(j\), the words \(\gamma_i^{(j)}\) are pairwise distinct
and have the same positive length, and the lengths strictly increase with
\(j\). If \(g'_k\neq0\), then
\(s_a^*g'_kp_{H_k}s_a=g'_kp_{r(H_k,a)}\neq0\), so
\(g'_kp_{H_k}s_a\neq0\). Since this term has degree \(1\), while the remaining
terms have strictly larger positive degree, the \(\mathbb Z\)-grading gives
\(\beta p_{H_k}s_a\neq0\).

If \(g'_k=0\), choose a nonzero homogeneous component
\(\sum_i k_i^{(j)}s_{\gamma_i^{(j)}}p_{H_k}\). Since
\(r(H_k,a)\neq\emptyset\), we have
\[
s_a^*s_{\gamma_1^{(j)}}^*
\left(
\sum_{i=1}^{n_j}k_i^{(j)}s_{\gamma_i^{(j)}}p_{H_k}s_a
\right)
=
k_1^{(j)}p_{r(H_k,a)}
\neq0.
\]
Thus that homogeneous component remains nonzero after multiplication by
\(s_a\), and again the grading gives \(\beta p_{H_k}s_a\neq0\). Hence
\(\alpha p_{H_k}s_a\neq0\). Its ghost degree is strictly smaller than \(d\), so
the induction hypothesis applies. This completes the proof.
\end{proof}

\section{Applications}

In this section, we apply Theorem~\ref{thm:reduction} to obtain uniqueness
theorems and two ring-theoretic consequences for Leavitt labelled path
algebras. More precisely, we prove the graded uniqueness theorem
(Theorem~\ref{thm:graded-uniq}) and the Cuntz--Krieger uniqueness theorem
(Theorem~\ref{thm:Cuntz-Krieger}), and we show that every Leavitt labelled
path algebra over a field is semiprime and semiprimitive.

We begin with a useful consequence of the reduction theorem for ideals.

\begin{lem}\label{ideals-vertex}
Let \((E,\mathcal L,\mathcal B)\) be a normal labelled space and \(K\) a field.
Then the following statements hold:
\begin{enumerate}
\item
For every nonzero graded ideal \(I\) of \(L_K(E,\mathcal L,\mathcal B)\), there
exists a nonempty set \(A\in\mathcal B\) such that \(p_A\in I\).

\item
If every cycle in \((E,\mathcal L,\mathcal B)\) has an exit, then for every
nonzero ideal \(I\) of \(L_K(E,\mathcal L,\mathcal B)\), there exists a
nonempty set \(A\in\mathcal B\) such that \(p_A\in I\).
\end{enumerate}
\end{lem}

\begin{proof}
Let \(I\) be a nonzero graded ideal and let \(0\neq\alpha\in I\). By
Theorem~\ref{thm:reduction}, there exist \(a,b\in L_K(E,\mathcal L,\mathcal B)\)
such that either \(a\alpha b=kp_A\), with \(0\neq k\in K\) and
\(\emptyset\neq A\in\mathcal B\), or
\(0\neq a\alpha b=\sum_{i=0}^n k_is_c^ip_A\), where \((c,A)\) is a cycle
without exit and \(s_c^0:=p_A\).

In the first case, \(kp_A\in I\), and hence \(p_A\in I\). In the second case,
since \(I\) is graded, each homogeneous component \(k_is_c^ip_A\) belongs to
\(I\). If \(k_0p_A\neq0\), then again \(p_A\in I\). Otherwise, for some
\(1\leq i\leq n\), we have \(k_is_c^ip_A\neq0\). By Lemma~\ref{cycle},
\((c^i,A)\) is a cycle, so \(A\subseteq r(A,c^i)\). Therefore
\[
p_A
=
k_i^{-1}(s_c^i)^*(k_is_c^ip_A)
\in I.
\]
This proves (1).

For (2), let \(I\) be a nonzero ideal and choose \(0\neq\alpha\in I\). Since
every cycle has an exit, the second alternative in Theorem~\ref{thm:reduction}
cannot occur. Hence \(a\alpha b=kp_A\) for some \(0\neq k\in K\) and
\(\emptyset\neq A\in\mathcal B\). Since \(a\alpha b\in I\), we get
\(p_A\in I\).
\end{proof}

The graded uniqueness theorem for Leavitt labelled path algebras was
established in \cite{BCGW:lpaolg} by Boava, de Castro, Gon\c{c}alves, and
van Wyk using the fact that every Leavitt labelled path algebra is a
Cuntz--Pimsner ring. We now recover this theorem as a consequence of the
reduction theorem.

\begin{thm}[{\cite[Corollary 5.5]{BCGW:lpaolg}}]\label{thm:graded-uniq}
Let \((E,\mathcal L,\mathcal B)\) be a normal labelled space, let \(K\) be a
field, let \(A\) be a \(\mathbb Z\)-graded \(K\)-algebra, and let
\(\phi:L_K(E,\mathcal L,\mathcal B)\to A\) be a \(\mathbb Z\)-graded
homomorphism. Then \(\phi\) is injective if and only if
\(\phi(p_B)\neq0\) for every nonempty \(B\in\mathcal B\).
\end{thm}

\begin{proof}
If \(\phi\) is injective, then \(\phi(p_B)\neq0\) for every nonempty
\(B\in\mathcal B\), because \(p_B\neq0\) by Proposition~\ref{properties}(1).

Conversely, suppose that \(\phi(p_B)\neq0\) for every nonempty \(B\in\mathcal B\).
If \(\ker(\phi)\neq0\), then \(\ker(\phi)\) is a nonzero graded ideal. By
Lemma~\ref{ideals-vertex}(1), there exists a nonempty \(B\in\mathcal B\) such
that \(p_B\in\ker(\phi)\). Hence \(\phi(p_B)=0\), a contradiction. Therefore
\(\ker(\phi)=0\), and \(\phi\) is injective.
\end{proof}

As another consequence of Theorem~\ref{thm:reduction}, we obtain the following
Cuntz--Krieger uniqueness theorem for Leavitt labelled path algebras, extending
\cite[Theorem 2.10]{dgn:otioulpa} to the labelled graph setting.

\begin{thm}\label{thm:Cuntz-Krieger}
Let \((E,\mathcal L,\mathcal B)\) be a normal labelled space such that every
cycle has an exit, let \(K\) be a field, let \(A\) be a \(K\)-algebra, and let
\(\phi:L_K(E,\mathcal L,\mathcal B)\to A\) be a \(K\)-algebra homomorphism.
Then \(\phi\) is injective if and only if \(\phi(p_B)\neq0\) for every nonempty
\(B\in\mathcal B\).
\end{thm}

\begin{proof}
The forward implication follows from Proposition~\ref{properties}(1). For the
converse, suppose that \(\phi(p_B)\neq0\) for every nonempty \(B\in\mathcal B\)
and that \(\ker(\phi)\neq0\). By Lemma~\ref{ideals-vertex}(2), there exists a
nonempty \(B\in\mathcal B\) such that \(p_B\in\ker(\phi)\). Hence
\(\phi(p_B)=0\), a contradiction. Therefore \(\ker(\phi)=0\), and \(\phi\) is
injective.
\end{proof}

We next use the reduction theorem to prove semiprimeness. Recall that a ring
\(R\) is {\it semiprime} if \(I^2=0\) implies \(I=0\), for every ideal \(I\) of \(R\).

\begin{thm}\label{semiprime}
Let \((E,\mathcal L,\mathcal B)\) be a normal labelled space and \(K\) a field.
Then \(L_K(E,\mathcal L,\mathcal B)\) is semiprime.
\end{thm}

\begin{proof}
Let \(I\) be a nonzero ideal of \(L_K(E,\mathcal L,\mathcal B)\), and choose
\(0\neq\alpha\in I\). By Theorem~\ref{thm:reduction}, either \(I\) contains a
nonzero scalar multiple of a projection \(p_A\), with
\(\emptyset\neq A\in\mathcal B\), or \(I\) contains a nonzero element
\[
x=\sum_{i=0}^n k_is_c^ip_A,
\]
where \((c,A)\) is a cycle without exit and \(s_c^0:=p_A\).

In the first case, \(p_A\in I\), and therefore
\(0\neq p_A=p_A^2\in I^2\). In the second case, choose \(n\) maximal with
\(k_n\neq0\). If \(n=0\), then \(x=k_0p_A\), so \(p_A\in I\), and again
\(I^2\neq0\). Assume now that \(n\geq1\). By Lemma~\ref{cycle}, the highest
homogeneous component of \(x^2\) is \(k_n^2s_c^{2n}p_A\). This term is
nonzero: if \(s_c^{2n}p_A=0\), then
\[
0=(s_c^{2n})^*s_c^{2n}p_A=p_{r(c^{2n})}p_A=p_A,
\]
because \(A\subseteq r(A,c^{2n})\subseteq r(c^{2n})\), a contradiction.
Hence \(x^2\neq0\), and so \(I^2\neq0\). Therefore
\(L_K(E,\mathcal L,\mathcal B)\) is semiprime.
\end{proof}

We finish by proving semiprimitivity. Recall that a ring \(R\) is
{\it semiprimitive} if its Jacobson radical \(J(R)\) is zero. An element \(x\in R\)
is {\it right quasi-regular} if there exists \(y\in R\) such that \(x+y-xy=0\), and
left quasi-regular if there exists \(z\in R\) such that \(x+z-zx=0\). An ideal
is {\it quasi-regular} if every element in it is both left and right quasi-regular.
The Jacobson radical is the largest quasi-regular ideal; see
\cite[Chapter I, Section 6, Theorem 1]{Jacobson}. The next theorem extends
\cite[Proposition 2.3.2]{AAS} and \cite[Theorem 2.12]{dgn:otioulpa} to the
labelled graph setting.

\begin{thm}\label{semiprimitive}
Let \((E,\mathcal L,\mathcal B)\) be a normal labelled space and \(K\) a field.
Then \(L_K(E,\mathcal L,\mathcal B)\) is semiprimitive.
\end{thm}

\begin{proof}
Let \(J\) be the Jacobson radical of \(L_K(E,\mathcal L,\mathcal B)\). Suppose
that \(J\neq0\), and choose \(0\neq\alpha\in J\). By
Theorem~\ref{thm:reduction}, there exist \(a,b\in L_K(E,\mathcal L,\mathcal B)\)
such that either \(a\alpha b=kp_A\), with \(0\neq k\in K\) and
\(\emptyset\neq A\in\mathcal B\), or
\(0\neq a\alpha b=\sum_{i=0}^n k_is_c^ip_A\), where \((c,A)\) is a cycle
without exit and \(s_c^0:=p_A\).

The first case is impossible, because it would imply \(p_A\in J\), and the
Jacobson radical contains no nonzero idempotents. In the second case, if only
the degree-zero term occurs, the same contradiction follows. Hence, replacing
\(a\alpha b\) by the obtained nonzero element of \(J\), we may assume that
\[
x=\sum_{i=0}^n k_is_c^ip_A\in J\setminus\{0\}
\]
with \(n\geq1\) and \(k_n\neq0\).

Since \(J\) is an ideal, \(s_c^m x\in J\) for all \(m\geq1\). Choosing \(m\)
large enough and reindexing the nonzero terms, we may replace \(x\) by a
nonzero element of \(J\) of the form
\[
x=\sum_{i=h}^n k_is_c^ip_A,
\]
where \(1\leq h\leq n\) and \(k_h,k_n\neq0\). By Lemma~\ref{cycle},
\(x=p_Axp_A\), so \(x\in J\cap p_AL_K(E,\mathcal L,\mathcal B)p_A\).

By \cite[Chapter III, Section 7, Proposition 1]{Jacobson},
\(J\cap p_AL_K(E,\mathcal L,\mathcal B)p_A\) is the Jacobson radical of the
corner \(p_AL_K(E,\mathcal L,\mathcal B)p_A\). Therefore, by
\cite[Chapter I, Section 6, Theorem 1]{Jacobson}, the element \(x\) is right
quasi-regular in this corner. Thus there exists a nonzero
\(y\in p_AL_K(E,\mathcal L,\mathcal B)p_A\) such that \(x+y-xy=0\).

Decompose \(y\) into its homogeneous components:
\(y=\sum_{i=k}^m y_i\), where \(k\leq m\), \(y_i\in
(p_AL_K(E,\mathcal L,\mathcal B)p_A)_i\), and \(y_m\neq0\). We first note that
\(s_c^np_Ay_m\neq0\). Indeed, if this product were zero, then
\[
0=(s_c^n)^*s_c^np_Ay_m=p_{r(c^n)}p_Ay_m=p_Ay_m=y_m,
\]
because \(A\subseteq r(A,c^n)\subseteq r(c^n)\), a contradiction.

Hence the homogeneous component of \(xy\) of degree \(n|c|+m\) is nonzero.
Indeed, it is \(k_ns_c^np_Ay_m\), and no term coming from
\(s_c^ip_Ay_j\), with \(i<n\) or \(j<m\), has degree \(n|c|+m\). If \(m>0\),
then \(n|c|+m>\max\{n|c|,m\}\), so this component cannot cancel with a
component of \(x+y\), contradicting \(x+y-xy=0\).

If \(m<0\), then every homogeneous component of \(xy\) has degree at most
\(n|c|+m<n|c|\), while every homogeneous component of \(y\) has degree at most
\(m<0\). Hence the degree \(n|c|\) homogeneous component of \(x+y-xy\) is
precisely \(k_ns_c^np_A\neq0\), again a contradiction. Thus \(m=0\).

Comparing the homogeneous component of degree \(n|c|\) in \(x+y-xy=0\), we
obtain \(k_ns_c^np_A=k_ns_c^np_Ay_0\). Since \(k_n\neq0\),
\(s_c^np_A=s_c^np_Ay_0\), and multiplying on the left by \((s_c^n)^*\) gives
\(p_A=p_Ay_0=y_0\).

If \(y=y_0=p_A\), then \(x+y-xy=x+p_A-xp_A=p_A\neq0\), a contradiction.
Therefore \(y\) has a nonzero homogeneous component of negative degree; let
\(k<0\) be the least degree with \(y_k\neq0\). We claim that
\(k_hs_c^hp_Ay_k\neq0\). If not, then
\[
0=k_h^{-1}(s_c^h)^*(k_hs_c^hp_Ay_k)=p_{r(c^h)}p_Ay_k=p_Ay_k=y_k,
\]
again a contradiction.

Thus the lowest homogeneous component of \(xy\) has degree \(h|c|+k\), which
is strictly greater than \(k\). On the other hand, the lowest homogeneous
component of \(x+y\) is \(y_k\), since all homogeneous components of \(x\) have
positive degree. Hence the degree \(k\) component of \(x+y-xy\) is
\(y_k\neq0\), contradicting \(x+y-xy=0\).

All possibilities lead to a contradiction, so \(J=0\). Therefore
\(L_K(E,\mathcal L,\mathcal B)\) is semiprimitive.
\end{proof}

%
%
%
%


\vskip 0.5cm

\begin{thebibliography}{99}

\bibitem{a:lpatfd}
G. Abrams,
Leavitt path algebras: the first decade,
\emph{Bull. Math. Sci.} \textbf{5} (2015), no.~1, 59--120.

\bibitem{ap:tlpaoag05}
G. Abrams and G. Aranda Pino,
The Leavitt path algebra of a graph,
\emph{J. Algebra} \textbf{293} (2005), no.~2, 319--334.

\bibitem{AAS}
G. Abrams, P. Ara and M. Siles Molina,
\emph{Leavitt Path Algebras},
Lecture Notes in Mathematics, vol.~2191, Springer, London, 2017.

\bibitem{AbHaz23}
G. Abrams and R. Hazrat,
Connections between abelian sandpile models and the \(K\)-theory of weighted
Leavitt path algebras,
\emph{Eur. J. Math.} \textbf{9} (2023), no.~2, Paper No.~21, 28 pp.

\bibitem{AbHaz25}
G. Abrams and R. Hazrat,
Monoids, dynamics and Leavitt path algebras,
\emph{Expo. Math.} \textbf{43} (2025), no.~5, Article No.~125684, 17 pp.

\bibitem{amp:nktfga}
P. Ara, M. A. Moreno and E. Pardo,
Nonstable \(K\)-theory for graph algebras,
\emph{Algebr. Represent. Theory} \textbf{10} (2007), 157--178.

\bibitem{BaCaPask}
T. Bates, T. M. Carlsen and D. Pask,
\(C^*\)-algebras of labelled graphs III---\(K\)-theory computations,
\emph{Ergodic Theory Dynam. Systems} \textbf{37} (2017), no.~2, 337--368.

\bibitem{BatesPask:caolg}
T. Bates and D. Pask,
\(C^*\)-algebras of labelled graphs,
\emph{J. Operator Theory} \textbf{57} (2007), no.~1, 207--226.

\bibitem{BaCaGoRo}
D. Bagio, G. Gil Canto, D. Gon\c{c}alves and D. Royer,
The reduction theorem for algebras of one-sided subshifts over arbitrary
alphabets,
\emph{Rev. R. Acad. Cienc. Exactas F\'is. Nat. Ser. A Mat. RACSAM}
\textbf{118} (2024), no.~2, Paper No.~72, 21 pp.

\bibitem{BCGW:lpaolg}
G. Boava, G. G. de Castro, D. Gon\c{c}alves and D. W. van Wyk,
Leavitt path algebras of labelled graphs,
\emph{J. Algebra} \textbf{629} (2023), 265--306.

\bibitem{BoavaCastroMortari:inverseSemigroups}
G. Boava, G. G. de Castro and F. de L. Mortari,
Inverse semigroups associated with labelled spaces and their tight spectra,
\emph{Semigroup Forum} \textbf{94} (2017), no.~3, 582--609.

\bibitem{BoavaCastroMortari:groupoidModels}
G. Boava, G. G. de Castro and F. de L. Mortari,
Groupoid models for the \(C^*\)-algebra of labelled spaces,
\emph{Bull. Braz. Math. Soc. (N.S.)} \textbf{51} (2020), 835--861.

\bibitem{CasWyk}
G. G. de Castro and D. W. van Wyk,
Labelled space \(C^*\)-algebras as partial crossed products and a simplicity
characterization,
\emph{J. Math. Anal. Appl.} \textbf{491} (2020), no.~1, Article No.~124290,
35 pp.

\bibitem{dgn:otioulpa}
T. T. H. Duyen, D. Gon\c{c}alves and T. G. Nam,
On the ideals of ultragraph Leavitt path algebras,
\emph{Algebr. Represent. Theory} \textbf{27} (2024), no.~1, 77--113.

\bibitem{CantoGon}
C. Gil Canto and D. Gon\c{c}alves,
Representations of relative Cohn path algebras,
\emph{J. Pure Appl. Algebra} \textbf{224} (2020), no.~7, Article No.~106310,
15 pp.

\bibitem{gon:ratrt19}
D. Gon\c{c}alves and D. Royer,
Representations and the reduction theorem for ultragraph Leavitt path algebras,
\emph{J. Algebraic Combin.} \textbf{53} (2021), no.~2, 505--526.

\bibitem{gr:saccfulpavpsgrt}
D. Gon\c{c}alves and D. Royer,
Simplicity and chain conditions for ultragraph Leavitt path algebras via
partial skew group ring theory,
\emph{J. Aust. Math. Soc.} \textbf{109} (2020), no.~3, 299--319.

\bibitem{GoncalvesRoyer:socleSubshift}
D. Gon\c{c}alves and D. Royer,
The socle of subshift algebras, with applications to subshift conjugacy,
\emph{Proc. Roy. Soc. Edinburgh Sect. A} \textbf{156} (2026), no.~3,
844--869.

\bibitem{HazNam1}
R. Hazrat and T. G. Nam,
Unital algebras being Morita equivalent to weighted Leavitt path algebras,
\emph{J. Algebraic Combin.} \textbf{62} (2025), no.~2, Paper No.~28, 21 pp.

\bibitem{HazNam2}
R. Hazrat and T. G. Nam,
On structural connections between sandpile monoids and weighted Leavitt path
algebras,
\emph{J. Algebra} \textbf{678} (2025), 543--569.

\bibitem{ima:tlpaou}
M. Imanfar, A. Pourabbas and H. Larki,
The Leavitt path algebras of ultragraphs,
\emph{Kyungpook Math. J.} \textbf{60} (2020), no.~1, 21--43.

\bibitem{Jacobson}
N. Jacobson,
\emph{Structure of Rings},
American Mathematical Society Colloquium Publications, vol.~37,
American Mathematical Society, Providence, RI, 1956.

\bibitem{NN:pisulpa}
T. G. Nam and N. D. Nam,
Purely infinite simple ultragraph Leavitt path algebras,
\emph{Mediterr. J. Math.} \textbf{19} (2022), no.~1, Paper No.~7, 20 pp.

\bibitem{tomf:auatelaacaastg03}
M. Tomforde,
A unified approach to Exel--Laca algebras and \(C^*\)-algebras associated to
graphs,
\emph{J. Operator Theory} \textbf{50} (2003), no.~2, 345--368.

\end{thebibliography}
\end{document}